\documentclass[11pt]{article}

 \usepackage{amsfonts,amssymb,amsmath,mathrsfs,mathtools,pifont}
\usepackage{bm}
\usepackage{color,latexsym,amsfonts}
\usepackage{graphicx}
\usepackage{lipsum}

\numberwithin{equation}{section}
\newcommand\blfootnote[1]{%
\begingroup
\renewcommand\thefootnote{}\footnote{#1}
\addtocounter{footnote}{-1}%
\endgroup
}

 \newcommand{\qed}{\hfill\rule{2mm}{3mm}\vspace{4mm}}

 \newtheorem{theorem}{Theorem}[section]
 \newtheorem{lemma}[theorem]{Lemma}
 
 \newtheorem{remark}{Remark}

 \def\beqnn{\begin{eqnarray*}}\def\eeqnn{\end{eqnarray*}}
 \def\<{\langle}\def\>{\rangle}

 \def\beqlb{\begin{eqnarray}}\def\eeqlb{\end{eqnarray}}
 \def\qed{\hfill$\Box$\medskip}

\begin{document}

\bigskip\bigskip

\noindent{\Large\bf Limit theorems for the minimal position of a branching \\
  random walk in random environment}\blfootnote{\noindent * Corresponding author.}\footnote{
\noindent  Supported by NSFC (NO.~11531001).
}

\noindent{%\normalsize\sf
Wenming Hong\footnote{ School of Mathematical Sciences
\& Laboratory of Mathematics and Complex Systems, Beijing Normal
University, Beijing 100875, P.R. China. Email: wmhong@bnu.edu.cn} ~ Wanting Hou\footnote{ Department of Mathematics, Northeastern University, Shenyang 110004, P.R. China. Email: 201531130022@mail.bnu.edu.cn} ~ Xiaoyue Zhang*\footnote{ School of Mathematical Sciences
\& Laboratory of Mathematics and Complex Systems, Beijing Normal
University, Beijing 100875, P.R. China. Email: zhangxiaoyue@mail.bnu.edu.cn}

\begin{center}
\begin{minipage}{12cm}
\begin{center}\textbf{Abstract}\end{center}
\footnotesize
We consider a branching system of random walk in random environment (in location) in $\mathbb{N}$. We will give the exact limit value of $\frac{M_{n}}{n}$,  where $M_{n}$ denotes the  minimal position of branching random walk at time $n$. A key step in the proof is to transfer our  branching random walks in random environment (in location) to  branching random walks in random environment (in time), by use of  Bramson's ``branching processes within a branching process" (\cite{BR78}).
\bigskip

\mbox{}\textbf{Keywords:}\quad Random walk in random environment, Branching random walk, Branching process, Minimal position, Jumping time; \\
\mbox{}\textbf{Mathematics Subject Classification}:  Primary 60J80; Secondary 60G50.

\end{minipage}
\end{center}

\section{ Introduction}

There have been abundant works on  the minimal position of branching random walks. Hammersley (\cite{HA74}), Kingman (\cite{KI75}), Biggins (\cite{BI76}) gave the first-order limit of the minimal position of a branching random walk.
When this model is extended to a random environment, both time and spatial position will affect the particle behavior. Therefore, from these two perspectives, it is possible to generate some different models of branching random walks in random environment. Greven and den Hollander (\cite{GH92}) considered the model where the reproduction law of the particles depends on their locations while the transition probabilities are the same everywhere. They discussed problems on global particle density and local particle density. Comets et al. (\cite{CMP98}) considered the model where both the reproduction law of the particles and  the transition probabilities depend on their locations in $\mathbb{N}$, and gave an appropriate classification of the transience and recurrence. Bartsch et al. (\cite{BG09}) considered the model which is similar to the model in (\cite{CMP98}) but the movement is limited to 0 and 1, they mainly cared about problems on local survival and global survival. Devulder (\cite{DA07}) consider a branching system of random walk in random environment (in location), where the particles branching with a fixed law but move as random walk in random environment, and particle's displacements are limited to $\pm1$. Devulder (\cite{DA07}) discussed the classification criteria for the case that the upper limit of $\frac{m_{n}^{*}}{n}$ is greater than $0$ and the case that the lower limit of $\frac{m_{n}^{*}}{n}$ is less than $0$, where $m_{n}^{*}$ denotes the location of the rightmost particle at time $n$. But the accurate velocity of limit of $\frac{m_{n}^{*}}{n}$ has not been specified. We also note that, Huang and Liu (\cite{HLL14}, \cite{HL14}) considered the branching random walks in random environment (in time) and gave the limit theorems of the minimal position and maximal position and some related large deviation principles.

In the present paper, we discuss a model that is similar to that in Devulder (\cite{DA07}) except that we restrict the displacements to be $0$ or $1$. We will give the exact limit value of $\frac{M_{n}}{n}$,  where $M_{n}$ denotes the  minimal position of branching random walk at time $n$. A key step is to transfer our  branching random walks in random environment (in location) to  branching random walks in random environment (in time), by use of  Bramson's ``branching processes within a branching process" (\cite{BR78}), and then applying the result of Huang and Liu (\cite{HLL14}, \cite{HL14}).

We consider a branching system of random walks in random environment (in location) in $\mathbb{N}$, where the particles branching with a fixed law but move as random walk in random environment. For details, let $(\omega_{i})_{i\in \mathbb{N}}$ be a collection of independent and identically distributed random variables, taking values in $(0,1)$. $\eta$ is the distribution of $\omega\coloneqq \left(\omega_{i}\right)_{i\in \mathbb{N}}$.  For any realization of the environment $\omega\coloneqq \left(\omega_{i}\right)_{i\in \mathbb{N}}$, we define a random walk in random environment $\left(X_{n}\right)_{n\in \mathbb{N}}$ which satisfies,
$X_{0}=0$, and
\beqlb\label{dis}
P_{\omega}(X_{n+1}=i|X_{n}=i)=1-P_{\omega}(X_{n+1}=i+1|X_{n}=i) =\omega_{i};
\eeqlb

We assume that there exists $\delta>0$ such that $\omega_{0}\in (\delta,1-\delta) $ $\eta$-$a.s.$ Based on this model of random walk, we construct the branching system as following,

$\bullet$ At time $n=0$, there is only one particle at the origin;

$\bullet$ At time $n=1$, the particle dies and reproduces $k$ offspring with probability $p_{k}$. Each particle moves independently to a new position as the way described in (\ref{dis});

$\cdots$ $\cdots$

$\bullet$ Iterating this procedure, each particle reproduces and makes displacement in the same way with their ancestors. And thus we get a branching  random walk in random environment (in location), write BRWiRE (in location) for short. To avoid the possibility of extinction and trivial special cases, we assume that
\beqlb\label{2.3}
p_{0}=0,~p_{1}<1.
\eeqlb
this implies  $m\coloneqq \displaystyle\sum_{k=0}^{\infty}kp_{k}>1$.

Let $Z_{n}$  represent the number of particles in generation $n$ of the BRWiRE (in location), with $X_{n; k}$, $k=1,\cdots, Z_n$, being the positions of these particles,
then $\{Z_{n}\}$ is a Galton-Watson process with $Z_{0}=1$, $P_{\omega}(Z_{1}=k)=p_{k}$. We call $P_{\omega}$ the {\it quenched law}, and if $\eta$ denotes the law of the environment $(\omega_{i})_{i\in \mathbb{N}}$, we call \beqnn
\mathbb{P}(\cdot)\coloneqq \int P_{\omega}(\cdot)\eta(dw),
\eeqnn
 the {\it annealed law}.

\

We write
\beqlb\label{M}
M_n:= \min_{1\leq k \leq Z_n} X_{n; k}.
\eeqlb
This model  is just the one considered by Bramson (\cite{BR78}) (and then by  Dekking and Host (\cite{DH91})) if the displacement transition probability in (\ref{dis}) is constant. We can also classify three different cases in terms of $\omega_{max}\coloneqq \sup\{x: x\in Supp~\omega_{0}\}$,
\beqnn
\omega_{max}
\begin{cases}
>\frac{1}{m},& \text{supercritical};\\
= \frac{1}{m},& \text{critical};\\
<\frac{1}{m},& \text{subcritical}.
\end{cases}
\eeqnn
which is  consistent with the classification in  Dekking and Host (\cite{DH91}). The limit behavior of $M_n$ can be obtained accordingly as follows.

\begin{theorem}\label{th2.1.}(Supercritical case)~~If $\omega_{max}>\frac{1}{m}$ then there exists an almost surely finite random variable $M$ such that
\beqnn
M_{n}\to M ~~~~\mathbb{P}\text{-a.s.}
\eeqnn
\end{theorem}

\begin{remark}
The result of Theorem \ref{th2.1.} is consistent with Theorem 3.1* in \cite{BG09}, indeed  the concept of local survival in \cite{BG09} is equivalent to finiteness of the limit random variable $M$. Our proof is based on the 0-1 law and with a full classification in Lemma \ref{le3.2.}.
\end{remark}

\begin{theorem}\label{th2.2.}(Subcritical case)~~If $\omega_{max}<\frac{1}{m}$ and $ m^{(2)}:=\displaystyle\sum_{k=0}^{\infty}k^{2}p_{k}<\infty$ then there exists a constant $\gamma>0$ satisfies
\beqlb\label{t2}
\frac{M_{n}}{n}\to\gamma ~~~~\mathbb{P}\text{-a.s.}
\eeqlb
where $\gamma=\frac{1}{\mathbb{E}\left[\frac{1}{1-m\omega_{0}e^{t_{+}}}\right]}$, and $t_{+}=\sup\big\{t<\log\frac{1}{m\omega_{max}}: \mathbb{E}[\frac{t}{1-m\omega_{0}e^{t}}]-\mathbb{E}[\log\frac{m(1-\omega_{0})e^{t}}{1-m\omega_{0}e^{t}}]\leq0\big\}.$
\end{theorem}

\begin{theorem}\label{th2.3.}(Critical case)~~If $\omega_{max}=\frac{1}{m}$ and $\eta(\omega_{0}=\frac{1}{m})>0$ then
\beqnn
M_{n}\to\infty~~and~~\frac{M_{n}}{n}\to0 ~~~~\mathbb{P}\text{-a.s.}
\eeqnn
\end{theorem}

\

\begin{remark} The proof of Theorem \ref{th2.1.} is based on the 0-1 law and the nondecreasing of $M_{n}$, which will figure out in section \ref{p1}. In section \ref{p2}, we will  prove Theorem \ref{th2.2.} and Theorem \ref{th2.3.}, where  a key step is to transfer our  branching random walks in random environment (in location) to  branching random walks in random environment (in time), by use of  Bramson's ``branching processes within a branching process" (\cite{BR78}), and then applying the result of Huang and Liu (\cite{HLL14}, \cite{HL14}).

\end{remark}

%\text{\bf Remark 1}: In (\cite{DH91}) Dekking and Host considered problems on minimal position of branching random walks with nonnegative displacements in non-random environment. It corresponds to the model in our passage with three cases : when $\omega_{0}\equiv C>\frac{1}{m}$, the limit of $M_{n}$ is limited and the author call this case supercritical branching random walk; when $\omega_{0}\equiv \frac{1}{m}$, $M_{n}$ goes to infinity but $\frac{M_{n}}{n}$ tends to $0$, this case is called critical branching random walk; when $\omega_{0}\equiv C<\frac{1}{m}$, $\displaystyle\lim_{n\to\infty}\frac{M_{n}}{n}=\gamma>0$, this case is called subcritical branching random walk. \\\\
\begin{remark}  Hammersley (\cite{HA74}), Kingman (\cite{KI75}), Biggins (\cite{BI76}) gave the first-order limit of the minimal position of a branching random walk in non-random environment. That is
\beqlb\label{t3}\frac{M_{n}}{n}\to\gamma_{1}
 \eeqlb
almost surely, where $\gamma_{1}=\inf\{a:\mu(a)\geq1\}$ with $\mu(a)=\inf\{e^{\theta a}\phi(\theta):\theta\geq0\}$, and $\phi(\theta)=E\displaystyle\sum_{|x|=1}e^{-\theta V(x)}$ satisfies $\phi(\theta)<\infty$ for some $\theta>0$ $(V(x)$ denotes the position of $x)$.

We will show that our result in (\ref{t2}) of Theorem  \ref{th2.2.} is consistent with the classical Hammersley-Kingman-Biggins Theorem for the branching random walk in non-random environment, i.e., when the displacement transition probability in (\ref{dis}) is a constant $\omega_{0}\equiv p<\frac{1}{m}$. To this end, on the one hand, $\phi(\theta)=mp+m(1-p)e^{-\theta}$, and
\beqnn
\mu(a)=
\begin{cases}
m,& \text{if $a\geq 1-p$};\\
\frac{mp}{1-a}[\frac{(1-p)(1-a)}{pa}]^{a},& \text{if $0<a<1-p$},\\
\end{cases}
\eeqnn
by Hammersley-Kingman-Biggins Theorem,
\beqlb\label{t31}
\gamma_{1}=\inf\Big\{a>0: \frac{mp}{1-a}[\frac{(1-p)(1-a)}{pa}]^{a}\geq1\Big\}.
\eeqlb
On the other hand, if $\omega_{0}\equiv p<\frac{1}{m}$, by use of the result in Theorem \ref{th2.2.},
\beqnn
t_{+}=\sup\Big\{t<\log\frac{1}{mp}: \frac{t}{1-mpe^{t}}-\log\frac{m(1-p)e^{t}}{1-mpe^{t}}\leq0\Big\},
\eeqnn
\beqnn
1-mpe^{t_{+}}&=& \inf\Big\{1-mpe^{t}>0: \frac{t}{1-mpe^{t}}-\log\frac{m(1-p)e^{t}}{1-mpe^{t}}\leq0\Big\}\\
&=& \inf\Big\{a>0: \frac{\log\frac{1-a}{mp}}{a}-\log\frac{(1-p)(1-a)}{ap}\leq0\Big\},
\eeqnn
thus
\beqlb\label{t21}
\gamma= \inf\Big\{a>0: \frac{\log\frac{1-a}{mp}}{a}-\log\frac{(1-p)(1-a)}{ap}\leq0\Big\}.
\eeqlb
Since $\frac{\log\frac{1-a}{mp}}{a}-\log\frac{(1-p)(1-a)}{ap}\leq0$ is equivalent to $\frac{mp}{1-a}[\frac{(1-p)(1-a)}{pa}]^{a}\geq1$, we get that $\gamma_{1}=\gamma$ by (\ref{t31}) and (\ref{t21}).

\end{remark}

\begin{remark} Compared with the model discussed by Devulder (\cite{DA07}), our model is simpler  but we give the exact limit value of $\frac{M_{n}}{n}$. It should be an interesting task to investigate  the accurate velocity of limit of $\frac{m_{n}^{*}}{n}$ for that of Devulder (\cite{DA07}), but a more complicated Bramson's ``branching processes within a branching process" (\cite{BR78}) should be constructed at first, which we are now going on.
\end{remark}

\begin{remark} Compared with the (one particle) RWRE driven by (\ref{dis}), the minimal position of the ``branching system" goes slowly.
Recall (\cite{ZE04}) the velocity of the RWRE being $\frac{1}{\mathbb{E}\frac{1}{1-\omega_{0}}}$, which is strictly bigger than $\gamma=\frac{1}{\mathbb{E}\frac{1}{1-m\omega_{0}e^{t_{+}}}}$ in (\ref{t2}) of Theorem  \ref{th2.2.}, the first-order limit of the minimal position of the BRWiRE (in location) because of  $t_{+}>0$ and $m>1$, as it should be.
\end{remark}

\section{0-1 law and the proof of  Theorem~\ref{th2.1.} \label{p1}}
First we give a classification criterion for the supercritical case.
\begin{lemma}\label{le3.1.}~~Let $\pi_{\omega}=P_{\omega}(M_{n}\to\infty)$, then for $\eta$-a.e. $\omega$, $\pi_{\omega}=0$ or $\pi_{\omega}=1$.
\end{lemma}
\begin{proof}
We denote by $P_{\omega}^{x}$ the law of the particle system conditionally on the environment $\omega$ and start from the position $x$ instead of $0$. $\theta$ is the shift operator, given by $(\theta\omega)_{i}\coloneqq\omega_{i+1}$. Then we have
\beqnn
P_{\omega}^{i}(M_{n}\to\infty)=P_{\theta^{i}\omega}(M_{n}\to\infty).
\eeqnn
Since $\omega_{i}$ is i.i.d, then sequence $\big\{P_{\theta^{i}\omega}(M_{n}\to\infty)\big\}_{i\in \mathbb{Z}}$ is a stationary sequence. Moreover, by a simple coupling argument, it is also a nondecreasing sequence. Thus it is constant, i.e.,
for $\eta$-a.e. $\omega$,
\beqlb\label{sta}
P_{\theta^{i}\omega}(M_{n}\to\infty)=P_{\omega}(M_{n}\to\infty),~\forall~i\in \mathbb{Z}.
\eeqlb
Let $M_{n}^{(j)}$ be the minimal displacement starting from the  $j^{th}$ particle in the first generation, $j=1,\cdots,Z_{1}$. Then
\beqnn
\{M_{n}\to\infty\}=\displaystyle\bigcap_{j=1}^{Z_{1}}\{M_{n}^{(j)}\to\infty\}.
\eeqnn
Let $N_{i}{(j)}$ be the number of particles at position $j$ at time $i$. Since $M_{n}^{(j)}$ are independent for different $j$, we have
\beqnn
P_{\omega}(M_{n}\to\infty)&=& E_{\omega}\Big[P_{\omega}^{1}(M_{n}\to\infty)^{N_{1}{(1)}}P_{\omega}^{0}(M_{n}\to\infty)^{N_{1}{(0)}}\Big]\\
&=& E_{\omega}P_{\omega}(M_{n}\to\infty)^{N_{1}{(1)}+N_{1}{(0)}}\\
&=& E_{\omega}P_{\omega}(M_{n}\to\infty)^{Z_{1}},
\eeqnn
the second equality is by (\ref{sta}). Recall the assumption $(\ref{2.3})$ we know that $Z_{1}\geq1$ and $Z_{1}>1$ with positive probability. Then $P_{\omega}(M_{n}\to\infty)=$$0$ or $1$, i.e.,
$
\pi_{\omega}=0~or~1,~\eta\text{-}a.e..
$
\qed
\end{proof}

\begin{lemma}\label{le3.2.}~~(i)~If $\omega_{max}>\frac{1}{m}$ then $\mathbb{P}(M_{n}\to\infty)=0$.

~~~~~~~~~~~~~~~(ii)~If $\omega_{max}\leq\frac{1}{m}$ then $\mathbb{P}(M_{n}\to\infty)=1$.
\end{lemma}
\begin{proof}
(i) Let $i_{\omega}\coloneqq\min\{j\geq0: m\omega_{j}>1\}$, $\mathcal{D}=\{\omega: P_{\omega}(i_{\omega}<\infty)=1\}$, $N_{i}(j)$ is the number of particles at position $j$ at time $i$. If $\omega_{max}>\frac{1}{m}$, then $\mathbb{P}(\mathcal{D})=1$. For every $\omega\in\mathcal{D}$, there exists $N_{\omega}$ satisfies
\begin{eqnarray}\label{3.1}
1-\pi_{\omega}&\geq& P_{\omega}(M_{n}=i_{\omega},~\forall~n\geq N_{\omega})\nonumber\\
&=& P_{\omega}\big(\displaystyle\lim_{n\to\infty}N_{n}(i_{\omega})>0\big)\nonumber\\
&>& 0 .
\end{eqnarray}
The last inequality is due to the fact that when a particle reaches $i_{\omega}$, this particle and its descendants which stay at $i_{\omega}$ form a Galton-Watson process with mean offspring $m\omega_{i_{\omega}}>1$, so it is a supercritical branching process and has positive probability to exist forever.

From Lemma~\ref{le3.1.} we know that $\pi_{\omega}=0$ or $1$, combined with (\ref{3.1}) we deduce that for any $\omega\in\mathcal{D}$, $\pi_{\omega}=0$, then
\beqnn
\mathbb{P}(M_{n}\to\infty)=\mathbb{E}\pi_{\omega}=0.
\eeqnn

(ii) When $\omega_{max}\leq\frac{1}{m}$, we need to prove that for $\mathbb{P}$-$a.e.$ $\omega$, $P_{\omega}(M_{n}\to\infty)=1$. Let $B\coloneqq\{\omega: P_{\omega}(M_{n}\to\infty)=0\}$. Suppose that $\mathbb{P}(B)>0$. Since
\beqnn
\{M_{n}\not\to\infty\}=\displaystyle\bigcup_{p}\big\{M_{n}\to p\big\}=\displaystyle\bigcup_{p}\displaystyle\bigcup_{M}\big\{n\geq M: M_{n}=p\big\},
\eeqnn
there exists $p_{\omega},~M_{\omega}$ such that
\beqnn
P_{\omega}(M_{n}=p_{\omega}: n\geq M_{\omega})>0.
\eeqnn
As a result we obtain that the sub-branching process which stays at $p_{\omega}$ is supercritical, that is to say $m\omega_{p_{\omega}}>1$, but $m\omega_{p_{\omega}}>1$ contradicts to the condition that $\omega_{max}\leq\frac{1}{m}$. So the assumption can not be true. Combined with Lemma~\ref{le3.1.} gives
\beqnn
P_{\omega}(M_{n}\to\infty)=1~\text{is true for}~\mathbb{P}\text{-}a.e.~\omega.
\eeqnn
Hence, $\mathbb{P}(M_{n}\to\infty)=1$.
\qed
\end{proof}

Lemma~\ref{le3.2.} tells us that only when $\omega_{max}>\frac{1}{m}$ can the limit of the minimal position of branching random walk be finite. In this case, we call the branching random walk supercritical.\\

\begin{proof} \textbf{of Theorem~\ref{th2.1.}} From Lemma~\ref{le3.2.} (i) we know that if $\omega_{max}>\frac{1}{m}$, $P_{\omega}(M_{n}\to\infty)=0$ $a.s.$ Since $M_{n}$ is nondecreasing, for any $\omega\in\mathcal{D}$ (defined in Lemma~\ref{le3.2.}), there exists an almost surely finite random variable $M(\omega)$ such that $M_{n}(\omega)\to M(\omega)$ $P_{\omega}$-$a.e.$, i.e.,
\beqnn
P_{\omega}\Big(M_{n}(\omega)\to M(\omega)\Big)=1,~\forall~\omega\in\mathcal{D}.
\eeqnn
For any $\omega\in\mathcal{D}^{c}$, let $M(\omega)=0$. Then we have $\mathbb{P}(M_{n}\to M)=\mathbb{E}P_{\omega}(M_{n}\to M)=1$ and $\mathbb{P}(M<\infty)=1$.
\qed
\end{proof}

\section{Proof of  Theorem~\ref{th2.2.} and  Theorem~\ref{th2.3.} \label{p2}}

 Firstly, when $\omega_{max}<\frac{1}{m}$, we can easily see that  $\displaystyle\liminf_{n\to\infty}\frac{M_{n}}{n}>0$,  $\mathbb{P}$-$a.s.$. Indeed, if we set $\rho\coloneqq(\omega_{max},\cdots,\omega_{max},\cdots)$, i.e., for any $i$, $(\rho)_{i}=\omega_{max}$, it is known (\cite{DH91}) that
\beqnn
P_{\rho}\Big(\displaystyle\lim_{n\to\infty}\frac{M_{n}}{n}=\gamma_{\rho}\Big)=1,
\eeqnn
where $\gamma_{\rho}>0$. By coupling method we conclude that for $\eta\text{-}a.e.~\omega$,
\beqnn
P_{\omega}\Big(\displaystyle\liminf_{n\to\infty}\frac{M_{n}}{n}>0\Big)\geq P_{\rho}\Big(\displaystyle\liminf_{n\to\infty}\frac{M_{n}}{n}>0\Big)=1.
\eeqnn
% Hence
 % \beqnn
 % \mathbb{P}\Big(\displaystyle\liminf_{n\to\infty}\frac{M_{n}}{n}>0\Big)=\mathbb{E}P_{\omega}\Big(\displaystyle\liminf_{n\to\infty}\frac{M_{n}}{n}>0\Big)=1.
% \eeqnn

\subsection {From a BRWiRE (in location) to a BRWiRE  (in time)}

Inspired by the method of  proving the law of large numbers for random walk in random environment (\cite{ZE04}), for exploring the limit behavior of $\frac{M_{n}}{n}$,  we will transfer our  branching random walks in random environment (in location) to  branching random walks in random environment (in time), by use of  Bramson's ``branching processes within a branching process" (\cite{BR78}).

In the model of branching random walk (in non-random environment) considered by Bramson (\cite{BR78}) (and then by  Dekking and Host (\cite{DH91})), i.e.,  the displacement transition probability in (\ref{dis}) is constant.
 Bramson (\cite{BR78}) intelligently proposed that $\{Y_{j}\}$ is also a branching process, where $\{Y_{j}\}$ refers to the number of particles that jump from location $j-1$ to $j$ at some time, the generating function of $\{Y_{j}\}$ is $\phi_{Y}(s)=\phi_{Z}\Big((1-p)s+p\phi_{Y}(s)\Big)$, where $1-p$ is the probability of the particle jumps one step up and $p$ is the probability that the particle stays in place (here $\phi_{W}$ denotes the generating function of the first generation distribution $W_{1}$ of the branching process $\{W_{j}\}$). For details, we need to introduce some notations,

$\bullet$ $X(a_{1},\cdots,a_{k})$ represents the relative displacement of the $a_{k}^{th}$ individual of the $k^{th}$ generation with forbears $(a_{1}),(a_{1},a_{2}),\cdots,(a_{1},\cdots,a_{k-1})$.

$\bullet$ $S(a_{1},\cdots,a_{k})$ represents the position of individual $(a_{1},\cdots,a_{k})$. Accordingly,
$S(a_{1},\cdots,a_{k})=\displaystyle\sum_{i=1}^{k}X(a_{1},\cdots,a_{i})$, $M_{n}=\displaystyle\min_{a_{1},\cdots,a_{n}}S(a_{1},\cdots,a_{n})$.

$\bullet$ $I_{j}=\{(a_{1},\cdots,a_{n}): S(a_{1},\cdots,a_{n-1})=j-1,~S(a_{1},\cdots,a_{n})=j\}$.

$\bullet$ $Y_{j}=|I_{j}|$, the cardinality of $I_{j}$.

For any $\nu\in I_{j}$, $|\nu|$ denotes the generation of $\nu$.

Let $\tau_{j1}\leq\tau_{j2}\leq\tau_{j3}\leq\cdots$ denote the generation of all individuals in $I_{j}$ and rank them in ascending order. Denote by \beqnn
L_{j}\coloneqq \max\{|\nu|: \nu\in I_{j}\}
\eeqnn
the latest generation time that particles jump from $j-1$ to $j$.

Bramson has already proved that $\{Y_{j}\}$ is a branching process, where $j$ represents location information in our original process. But now we need to take a different perspective to view $j$ as time, and treat $|\nu|(\nu\in I_{j})$, information originally representing time, as the location information of the new branching random walk, that is to say, when we care about these quantities of $\tau_{ji}$, we get a new branching random walk. This new branching random walk can be considered as being constructed as follows,

$\bullet$ At time $0$, there is one particle $\varnothing$ at the origin;

$\bullet$ At time $1$, this particle splits into a random number $Y(\varnothing)$ particles, and these particles move to the position $\tau_{11}, \tau_{12}, \cdots, \tau_{1Y(\varnothing)}$, where $\tau_{1l}$ are integer-valued random variables (may not be independent of each other) and the distribution of the random vector $X(\varnothing)\coloneqq\big(Y(\varnothing), \tau_{11}, \tau_{12}, \cdots, \tau_{1Y(\varnothing)}\big)$ is $\xi_{0}=\xi(\omega_{0})$ (when given the environment $\omega$), which is determined by
\begin{eqnarray}\label{m0t}
m_{0}(t) &\coloneqq& E_{\omega}\Big[\displaystyle\sum_{l=1}^{Y(\varnothing)}e^{t\tau_{1l}}\Big]= E_{\omega}\Big[\displaystyle\sum_{i=1}^{\infty}Y(\varnothing,i)e^{ti}\Big]\nonumber \\
&=& \displaystyle\sum_{i=1}^{\infty}m^{i}\omega_{0}^{i-1}(1-\omega_{0})e^{ti} ,
\end{eqnarray}
where $Y(\varnothing,i)$ represents the number of particle $\varnothing$'s children which locate in position $i$.

$\cdots$ $\cdots$

$\bullet$ At time $n$, the particle $\nu$ ($|\nu|=n-1)$ located at position $k$ splits into a random number $Y(\nu)$ particles, and these particles move to $\tau_{n1}, \tau_{n2}, \cdots, \tau_{nY(\nu)}$, where the distribution of $(Y(\nu), \tau_{n1}-k, \tau_{n2}-k, \cdots, \tau_{nY(\nu)}-k)$ is $\xi_{n-1}=\xi(\omega_{n-1})$, which is determined by
\begin{eqnarray}\label{mnt}
m_{n-1}(t) &\coloneqq& E_{\omega}\Big[\displaystyle\sum_{l=1}^{Y(\nu)}e^{t(\tau_{nl}-k)}\Big]= E_{\omega}\Big[\displaystyle\sum_{i=1}^{\infty}Y(\nu,i)e^{ti}\Big]\nonumber \\
&=& \displaystyle\sum_{i=1}^{\infty}m^{i}\omega_{n-1}^{i-1}(1-\omega_{n-1})e^{ti} ,
\end{eqnarray}
where $Y(\nu,i)$ represents the number of particle $\nu$'s children which locate in position $k+i$.

$\bullet$ Iterating this procedure and we get a new branching random walk with a random environment in time. \\

If we use $Y(\nu,i)$ to denote the number of particle $\nu$'s children which locate in position $i+V(\nu)$ $(V(x)$ denotes the position of $x)$, then $Y(\nu)= \displaystyle\sum_{i=1}^{\infty}Y(\nu,i)$. By the structure of our model, we know that the descendants of $\varnothing$ that stay at 0 form a branching process with mean offspring $m\omega_{0}$, we use $\{N_n(0)\}$ and $\phi_{N(0)}(s)$ to denote this process and its generating function respectively. From its branching structure we have $\phi_{N(0)}(s)=\phi_{Z}(\omega_{0}s+1-\omega_{0})$.\\

The corresponding relationship between first two generations of the new branching random walk and the previous one is shown in the figure,
\begin{figure}[h]
\begin{center}
\includegraphics[width=16cm]{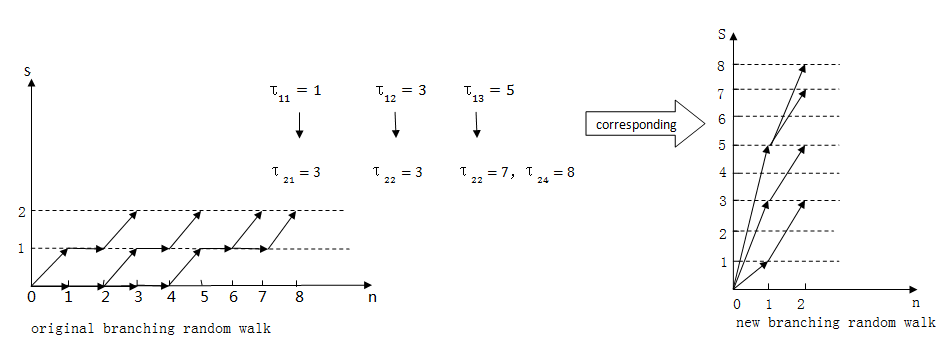}
\end{center}
\end{figure}

Then $Y_{j}$  represent the number of particles in generation $j$ of the BRWiRE (in time), with $S_{j; k}$, $k=1,\cdots, Y_j$, being the positions of these particles, and
\beqlb\label{L}
L_j:= \max_{1\leq k \leq Y_j} S_{j; k}.
\eeqlb

\subsection {Relationship between $M_{n}$ in (\ref{M}) and  $L_{n}$ in (\ref{L})}

\begin{lemma}\label{LeR} If $\displaystyle\lim_{n\to\infty}\frac{L_{n}}{n}=\alpha>0$,  ~$\mathbb{P}\text{-a.s.}$ then  $\displaystyle\lim_{n\to\infty}\frac{M_{n}}{n}=\frac{1}{\alpha}$, ~$\mathbb{P}\text{-a.s.}$.
\end{lemma}

\begin{proof} Take $k_{n}$ as a unique integer to satisfy,
\begin{eqnarray}\label{kni}
L_{k_{n}}\leq n<L_{k_{n}+1}.
\end{eqnarray}
Recalling the definition of $L_{n}$ we have $k_{n}\leq M_{n}<k_{n}+1$, i.e. $\frac{k_{n}}{n}\leq\frac{M_{n}}{n}<\frac{k_{n}+1}{n}$. As a consequence
\beqnn
\displaystyle\lim_{n\to\infty}\frac{M_{n}}{n}=\displaystyle\lim_{n\to\infty}\frac{k_{n}}{n}.
\eeqnn
Given the condition $\displaystyle\lim_{n\to\infty}\frac{L_{n}}{n}=\alpha$, then
\begin{eqnarray}\label{lkn}
\displaystyle\lim_{n\to\infty}\frac{L_{k_{n}}}{k_{n}}=\alpha~~~\text{and}~~~\displaystyle\lim_{n\to\infty}\frac{L_{k_{n}+1}}{k_{n}+1}=\alpha,
\end{eqnarray}
since $k_{n}\to\infty$ as $n\to\infty$. Combining (\ref{kni}) and (\ref{lkn}) we get
\beqnn
\displaystyle\lim_{n\to\infty}\frac{n}{k_{n}}\geq\alpha~~~\text{and}~~~\displaystyle\lim_{n\to\infty}\frac{n}{k_{n}+1}\leq\alpha.
\eeqnn
As a result, we obtain
\beqnn
\displaystyle\lim_{n\to\infty}\frac{n}{k_{n}}=\alpha~~~\text{and}~~~\displaystyle\lim_{n\to\infty}\frac{M_{n}}{n}=\frac{1}{\alpha}.
\eeqnn
\qed

\end{proof}

\subsection {Proof of Theorem~\ref{th2.2.}}

Let
\beqnn
\Lambda(t)\coloneqq\mathbb{E}[\text{log}~m_{0}(t)],
\eeqnn
\beqnn
B\coloneqq\{t:\Lambda(t)<\infty\}.
\eeqnn
In order to use the result in branching random walk in a random environment (in time) (\cite{HLL14},\cite{HL14}), we have to ensure the following conditions (1)-(4):
\beqnn
(1) \mathbb{E}\log m_{0}(0)\in (0,\infty),~~(2)0\in \mathring{B},~~(3) \mathbb{E}\tau_{11}<\infty,~~(4) \mathbb{E}\frac{Y}{m_{0}(0)}\log^{+}Y<\infty.
\eeqnn
\textbf{Condition~(1)} Since $\omega_{max}<\frac{1}{m}$, $m>1$, use ($\ref{m0t}$) we can calculate that
$
\mathbb{E}\log m_{0}(0)=\mathbb{E}\big[\log\frac{m(1-\omega_{0})}{1-m\omega_{0}}\big]\in(0,\infty).
$\\
\textbf{Condition~(2)} If $t>\log\frac{1}{m\omega_{max}}$, i.e. $m\omega_{max}e^t>1$, then there exists a set A satisfies $\eta(A)>0$ and for any $\omega\in A, m\omega e^t\geq1$, therefore from ($\ref{m0t}$) we see that $m_{0}(t)=\infty$ on A. $\Lambda(t)=\int_{A}\log m_{0}(t)d\eta+\int_{A^{c}}\log m_{0}(t)d\eta=\infty$; if $t<\log\frac{1}{m\omega_{max}}$, then there exists a constant c such that $m\omega_{max}e^t\leq c<1$, in that case $\Lambda(t)=\mathbb{E}[\log\frac{m(1-\omega_{0})e^{t}}{1-m\omega_{0}e^{t}}]\leq \mathbb{E}[\log\frac{me^{t}}{1-c}]=\log\frac{me^t}{1-c}<\infty.$

Therefore $\mathring{B}=\big\{t:t<\log\frac{1}{m\omega_{max}}\big\}$. Since $\log\frac{1}{m\omega_{max}}>0$, $0\in \mathring{B}$ obviously.\\
\textbf{Condition~(3)} Note that $\omega_{max}<\frac{1}{m}$ ensures that $\{N_{n}(0)\}$ is a subcritical branching process and $\mathbb{E}T_{0}<\infty$, where $T_{0}$ denotes the extinction time of the branching process that always stays at 0, then $\mathbb{E}\tau_{11}\leq\mathbb{E}T_{0}<\infty.$
\\
\textbf{Condition~(4)} Let $N_{\infty}(0)$ denotes the total population of $\{N_{n}(0)\}$. We number all the particles in $\{N_{n}(0)\}$ from 1 to $N_{\infty}(0)$, let $N(0,i)$ represents the number of offspring of the $i^{th}$ particle of $\{N_{n}(0)\}$ that jumps from 0 to 1. When given the environment $\omega$, $N(0,i)$ are i.i.d., the generating function of $N(0,i)$ is $\phi_{N(0,i)}(s)=\phi_{Z}(\omega_{0}+(1-\omega_{0})s).$
\begin{eqnarray}\label{oY2}
E_{\omega}Y^{2}=E_{\omega}\big(\displaystyle\sum_{i=1}^{N_{\infty}(0)}N(0,i)\big)^{2} &=& E_{\omega}\big[E_{\omega}[(\displaystyle\sum_{i=1}^{N_{\infty}(0)}N(0,i))^{2}|N_{\infty}(0)]\big]\nonumber\\
&=& E_{\omega}\Big[E_{\omega}\big[\displaystyle\sum_{i=1}^{N_{\infty}(0)}N(0,i)^{2}+\displaystyle\sum_{1\leq i\neq j \leq N_{\infty}(0)}N(0,i)N(0,j)|N_{\infty}(0)\big]\Big]\nonumber\\
&=& E_{\omega}\big[\displaystyle\sum_{i=1}^{N_{\infty}(0)}E_{\omega}[N(0,i)^{2}]+\displaystyle\sum_{1\leq i\neq j \leq N_{\infty}(0)}E_{\omega}[N(0,i)N(0,j)]\big]\nonumber\\
&\leq& E_{\omega}[N_{\infty}(0)]E_{\omega}[N(0,1)^{2}]+E_{\omega}[(N_{\infty}(0))^{2}][E_{\omega}(N(0,1))]^{2}
\end{eqnarray}

Use the expression of $\phi_{N(0,1)}(s)=\phi_{Z}(\omega_{0}+(1-\omega_{0})s)$ and $\phi_{N(0)}(s)=\phi_{Z}(\omega_{0}s+1-\omega_{0})$, we can calculate that $E_{\omega}N(0,1)=m(1-\omega_{0})$, $E_{\omega}[N(0,1)^{2}]=(m^{(2)}+m)(1-\omega_{0})^{2}+m(1-\omega_{0})$, $E_{\omega}N(0)=m\omega_{0}<m\omega_{max}<1$, $E_{\omega}N(0)^{2}=(m^{(2)}+m)\omega_{0}^{2}+m\omega_{0}$. Then from ($\ref{oY2}$),
\begin{eqnarray}\label{oY3}
E_{\omega}Y^{2} &\leq& E_{\omega}N_{\infty}(0)(m^{(2)}+2m)+E_{\omega}[N_{\infty}(0)^{2}]m^2 \nonumber\\
&\leq& E[N_{\infty}(0)|\omega_{0}=\omega_{max}](m^{(2)}+2m)+E[N_{\infty}(0)^{2}|\omega_{0}=\omega_{max}]m^2
\end{eqnarray}
Since from our assumption $m\omega_{max}<1$, $m^{(2)}<\infty$, $E[N(0)^{2}|\omega_{0}=\omega_{max}]<\infty$. From Lemma 3.1 in \cite{DP96} we have $E[N_{\infty}(0)^{2}|\omega_{0}=\omega_{\max}]<\infty$, $E[N_{\infty}(0)|\omega_{0}=\omega_{max}]<\infty$. Combined with (\ref{oY3}), $\mathbb{E}Y^{2}=\int E_{\omega}Y^{2}d\eta\leq E[N_{\infty}(0)|\omega_{0}=\omega_{max}](m^{(2)}+2m)+E[N_{\infty}(0)^{2}|\omega_{0}=\omega_{max}]m^2<\infty$, thus condition (4) satisfies obviously.

For $t\in\mathring{B}$, we have
\beqnn
\Lambda(t)\coloneqq\mathbb{E}[\log m_{0}(t)]=\mathbb{E}\Big[\log\frac{m(1-\omega_{0})e^{t}}{1-m\omega_{0}e^{t}}\Big]<\infty ,
\eeqnn
\beqnn
\Lambda^{'}(t)=\mathbb{E}\big[\frac{m_{0}^{'}(t)}{m_{0}(t)}\big]=\mathbb{E}\big[\frac{1}{1-m\omega_{0}e^{t}}\big]<\infty .
\eeqnn

Let
\beqnn
\rho(t)\coloneqq t\Lambda^{'}(t)-\Lambda(t), ~~~~t\in\mathring{B}.
\eeqnn
\beqnn
t_{+}\coloneqq\sup\{t\in\mathring{B}: t\Lambda^{'}(t)-\Lambda(t)\leq0\}.
\eeqnn
Notice that $\rho '(t)=t\Lambda^{''}(t)$, and $\Lambda^{''}(t)\geq0~\text{for}~t\geq0$. Therefore, $\rho(t)$ increases on $[0,\infty)$. Since $\rho(0)=-\Lambda(0)<0$, $\rho(t)$ is continuous on $B$, we obtain that $t_{+}>0$.

From Theorem 3.4 in \cite{HL14} we get the result that $\displaystyle\lim_{m\to\infty}\frac{L_{m}}{m}=\Lambda^{'}(t_{+})=\mathbb{E}[\frac{1}{1-m\omega_{0}e^{t_{+}}}]$, where
$t_{+}=\sup\big\{t<\log\frac{1}{m\omega_{max}}: \mathbb{E}[\frac{t}{1-m\omega_{0}e^{t}}]-\mathbb{E}[\log\frac{m(1-\omega_{0})e^{t}}{1-m\omega_{0}e^{t}}]\leq0\big\}$. As a consequence,   the theorem follows from Lemma \ref{LeR}.
\qed

\begin{remark}  We require $m^{(2)}<\infty$ to simplify our proof to ensure $\mathbb{E}Y^{2}<\infty$, so that condition (4) is satisfied, this assumption may be relaxed since condition(4) is much weaker than $\mathbb{E}Y^{2}<\infty$.\\
\end{remark}

\subsection{Proof of  Theorem~\ref{th2.3.} }

\begin{proof} \textbf{of Theorem~\ref{th2.3.}} $M_{n}\to\infty~\mathbb{P}$-$a.s.$ has already been proved in Lemma~\ref{le3.2.}, we only need to prove that $\frac{M_{n}}{n}\to0~\mathbb{P}$-$a.s.$ We still use the idea to transform the original branching random walk into a new branching random walk with random environment in time. In this situation, we can not use (\cite{HL14}) any more since $\mathbb{E}|\log m_{0}(0)|=\infty$, but Lemma \ref{LeR} is still valid i.e.
\beqnn
\displaystyle\lim_{n\to\infty}\frac{L_{n}}{n}=\displaystyle\lim_{n\to\infty}\frac{n}{M_{n}}
\eeqnn
is still valid if the limit of $\frac{L_{n}}{n}$ exists. Thus we just need to show that
\beqnn
\displaystyle\lim_{n\to\infty}\frac{L_{n}}{n}=\infty~~~~\mathbb{P}\text{-}a.s.
\eeqnn

Since in the new branching random walk, the position of the first generation individual is the time when the original branching random walk individual jumps from $0$ to $1$, then the maximum value of the first generation's position is the maximum moment when the original branching random walk jumps from $0$ to $1$. On account of $p_{0}=0$ and the descendants of $\varnothing$ that stay at $0$ form a Galton-Watson process with mean offspring $m\omega_{0}$, the maximal jumping moment is the extinction time of this Galton-Watson process, that is,

Let $W_{k}^{(u)}$ denote the Galton-Watson process formed by particle $u$ and its descendants that stay at $V(u)$ \big($V(u)$ denotes the location of $u$\big). And let
\beqnn
T_{0}\coloneqq\{n: W_{n-1}^{(\varnothing)}\geq1,~W_{n}^{(\varnothing)}=0\}.
\eeqnn
Thus the maximum relative displacement of particles from $\varnothing$ is
\beqnn
\displaystyle\max_{1\leq i\leq Y(\varnothing)}\tau_{1i}=T_{0}.
\eeqnn

If $u_{1}$ denotes the particle that jumps from $0$ to $1$ at time $T_{0}$, similar to previous analysis we know that the descendants of $u_{1}$, that stay at $1$ form a Galton-Watson process with mean offspring $m\omega_{1}$.

Denote
\beqnn
T_{1}\coloneqq\{n: W_{n-1}^{(u_{1})}\geq1,~W_{n}^{(u_{1})}=0\}.
\eeqnn
Then the maximum relative displacement of particles from $u_{1}$ is
\beqnn
\displaystyle\max_{1\leq i\leq Y(u_{1})}(\tau_{2i}-T_{0})=T_{1},
\eeqnn
where $\tau_{2i}(1\leq i\leq Y(u_{1}))$ denotes the position of the descendants of $u_{1}$.

Iterating this procedure we obtain
\beqnn
L_{n}\geq\displaystyle\sum_{i=0}^{n-1}T_{i}.
\eeqnn
Hence,
\begin{eqnarray}\label{lnf}
\frac{L_{n}}{n}\geq\frac{1}{n}\displaystyle\sum_{i=0}^{n-1}T_{i}.
\end{eqnarray}
where $T_{i}$ is the extinction time of a Galton-Watson process with mean offspring $m\omega_{i}$. Recalling that $\omega_{i}$ is i.i.d, then $T_{i}$ is also i.i.d.

Besides,
\beqnn
\mathbb{E}(T_{0})=\displaystyle\int_{\{\omega_{0}=\frac{1}{m}\}}E_{\omega}(T_{0})d\eta + \displaystyle\int_{\{\omega_{0}<\frac{1}{m}\}}E_{\omega}(T_{0})d\eta.
\eeqnn
Note that on the set $\{\omega_{0}=\frac{1}{m}\}$, $E_{\omega}(T_{0})=\infty$, and on the set $\{\omega_{0}<\frac{1}{m}\}$, $0<E_{\omega}(T_{0})<\infty$. Combined with the assumption $\eta\big(\{\omega_{0}=\frac{1}{m}\}\big)>0$, we get that $\mathbb{E}(T_{0})=\infty$. By (\cite{DR04}) Theorem (7.2)
\beqnn
\frac{1}{n}\displaystyle\sum_{i=0}^{n-1}T_{i}\to\infty~~\mathbb{P}\text{-}a.s.,~as~n\to\infty.
\eeqnn
(\ref{lnf}) suggests that $\displaystyle\lim_{n\to\infty}\frac{L_{n}}{n}=+\infty~~\mathbb{P}$-$a.s.$ and the conclusion follows from Lemma \ref{LeR}.

\qed
\end{proof}

%\noindent\bf{\footnotesize Acknowledgements}\quad\rm {\footnotesize We should thank professor  Vladimir Vatutin tell us the paper \cite{BG09} }

\end{document}